\documentclass[11pt]{article}
\usepackage{indentfirst}
\usepackage{amsmath,amssymb,amsfonts,amsthm,graphics}
\usepackage{makeidx}
\usepackage{indentfirst,graphics,epsfig,psfrag}
\usepackage{ifpdf}
\usepackage{color}
\newtheorem{thm}{Theorem}

\newtheorem{lem}{Lemma}[section]
\newtheorem{cla}{Claim}

\theoremstyle{definition}

\def\-{\mbox{--}}

\begin{document}
\title{\Large\bf A note on the $3$-rainbow index of $K_{2,t}$ }
\author{
\small  Tingting Liu, Yumei Hu \footnote{supported by NSFC No. 11001196. }  \\
\small  Department of Mathematics, Tianjin University, Tianjin 300072, P. R. China\\
\small E-mails: ttliu@tju.edu.cn; huyumei@tju.edu.cn;
 }
\date{}
\maketitle
\begin{abstract}
A tree $T$, in an edge-colored  graph $G$, is called {\em a rainbow tree} if no two edges of $T$ are assigned the same color. For a vertex
subset $S\in V(G)$, a tree that connects $S$ in $G$ is called an $S$-tree. A  {\em $k$-rainbow coloring } of $G$ is an edge coloring of $G$ having the property that for every set $S$ of $k$ vertices of $G$, there exists a rainbow $S$-tree $T$ in $G$. The minimum number of colors needed in a $k$-rainbow coloring of $G$ is the {\em $k$-rainbow  index of $G$}, denoted by $rx_k(G)$. In this paper, we obtain the exact values of $rx_3(K_{2,t})$ for
 any $t\geq 1$.


\end{abstract}

\section{Introduction}

All graphs considered in this paper are simple, finite and
undirected. We follow the terminology and notation of Bondy and
Murty \cite{bondy2008graph}. Let $G$ be a nontrivial connected graph
of order $n$ on which is defined an edge coloring, where adjacent
edges may be the same color. A path $P$ is a {\em rainbow path} if no two edges of
$P$ are colored the same. The graph $G$ is {\em rainbow connected }if
$G$ contains a $u-v$ rainbow path for every pair $u,v$ of distinct
vertices of $G$. The minimum number of colors that results in a
rainbow connected graph $G$ is the {\em rainbow connection number} $rc(G)$ of $G$.
These concepts were introduced by Chartrand et al. in \cite{Chartrand2008graph}.


Another generalization of rainbow connection number was also introduced by Chartrand et al. \cite{Chartrand2009graph}.
A tree $T$ is a {\em rainbow tree} if no two edges of $T$
are colored the same.  For a vertex
subset $S\in V(G)$, a tree that connects $S$ in $G$ is called an $S$-tree.
 Let $k$ be a fixed integer with $2\leq k\leq n$.
An edge coloring of $G$ is called a {\em k-rainbow coloring} if for every set
$S$ of $k$ vertices of $G$, there exists a rainbow $S$-tree. The {\em k-rainbow index} $rx_k(G)$ of $G$ is the minimum
number of colors needed in a $k$-rainbow  coloring of $G$. It is obvious that
$rc(G)=rx_2(G)$.
It follows, for every nontrivial connected graph $G$ of order $n$, that
$$ rx_2(G)\leq rx_3(G)\leq \cdots \leq rx_k(G).$$

Chakraborty et al. \cite{chakraborty2009hardness} showed
that computing the rainbow connection number of a graph is
NP-hard. Thus, it is more difficult to compute $k$-rainbow index of
general graphs.

For complete bipartite graph, Chartrand et al. \cite{Chartrand2008graph} obtained $rc(K_{s,t})=min\{\sqrt[s]{t},4\}$, for integers s and t with $ 2\leq s \leq t$.
 More results on the rainbow connection number can be found in the survey \cite{LiSun}. For $3$-rainbow index,
Li et al. \cite{Li2013} obtained the exact value of regular complete bipartite $K_{r,r}$, $rx_3(K_{r,r}) = 3$, with $r\geq 3$.

In \cite{Hu2013}, we showed, for any  integers s and t with $3\leq s \leq t$,
$rx_3(K_{s,t})\leq min \{6,s+t-3\}$, and the bound is tight.
 But this bound can not be generalized to the  graph $K_{2,t}$.
 So in the paper, we derive the exact value of  $rx_3(K_{2,t})$ for different $t(t\geq 1)$. We get the following theorem.
\begin{thm}\label{thm1}
For any integer $t\geq 1$,
 \[
      rx_3(K_{2,t})=\left\{
      \begin {array}{lll}

      2, &\mbox{ if $t=1,2$;}\\
      3, &\mbox{ if $t=3,4$;}\\
      4, &\mbox{ if $5\leq t\leq8$;}\\
      5, &\mbox{ if $9\leq t\leq 20$;}\\
      k, &\mbox{ if $(k-1)(k-2)+1\leq t\leq k(k-1)$,~($k\geq 6$);}\\
      \end {array}
      \right.
\]
\end{thm}

\section{Proof of Theorem \ref{thm1} }

In this section, we determine the  $3$-rainbow index of  complete bipartite graphs $K_{2,t}$. First of all, we need some new techniques and notions.

Let $U$ and $W$ be the two partite sets of $K_{2,t}$, where $U=\{u_1,~u_2\}, W=\{w_1,~w_2,~\cdots,w_t\}$. Suppose that there exists a
$3$-rainbow coloring $c$ : $E(K_{2,t})$ $\longrightarrow$ $\{1,2,\cdots, k \}$. Corresponding to the  $3$-rainbow coloring, there is a color  code($w$) assigned to every vertex $w\in W$,  consisting of an ordered $2$-tuple ($a_1,~a_2$), where $a_i=c(u_iw)\in\{1,2,\cdots,k\}$ for $i=1,2$.
In turn, for a subset $Y$ of $W$, given color codes of  vertices in $Y$ are  \textit{acceptable} if the corresponding coloring  is $3$-rainbow coloring of the graph induced by $Y\cup U$.
Let $B$ be a set of colors.  Color codes are $B$-\textit{limited} if both colors in every color code, but not necessarily distinct, are
from $B$. The  maximum number of   color codes which are not only $B$-limited but also acceptable
is denoted by $ \beta_{B}$.

Note that we adopt the following thought in the proof: we first give a certain $B$ with $k$ colors, then we consider the the maximum number of   color codes which are not only $B$-limited but also acceptable, the number is the tight upper bound of $t$ with $rx_3(K_{2,t})=k$.

The following claims are easy to verify and will be used later.

\begin{cla}\label{cla1}
If $|B|$=1, then $\beta_{B}\leq 1$.
\end{cla}

\begin{cla}\label{cla2}
If $|B|$=2, then $\beta_{B}\leq 2$.
\end{cla}
\begin{proof}
By contradiction. We may assume $\beta_{B}\geq 3$. For three vertices in $W$,  we can find a rainbow tree containing them. We know the rainbow tree containing them   uses at least an edge adjacent with every vertex of them, thus the tree uses at least three edges whose coloring are from $B$. Since the color codes are $B$-limited and $|B|$=2, the tree is not a rainbow tree, a contradiction.
\end{proof}

\begin{lem}\label{lem1}
For $t=1,2,$ $rx_3(K_{2,t})=2$, and $rx_3(K_{2,t})\geq 3$ for $t\geq 3$.
\end{lem}

\begin{proof}
Since $K_{2,1}$ is a tree, $rx_3(K_{2,1})=2$. For $t=2$,
$rx_3(K_{2,2})=rx_3(C_4)=2$. From the Claim \ref{cla2}, we get if $t\geq 3$, then $rx_3(K_{2,t})\geq 3$.
\end{proof}

The following lemma reminds us how to construct  color codes to some extent and is useful to show that an edge coloring is $3$-rainbow coloring by character of     color codes.
\begin{lem}\label{lem2}
Let $c$ be an edge coloring of $K_{2,t}$ with $rx_3(K_{2,t})=k$ and $S=\{v_1,v_2,v_3\}$ be any a set of three vertices in $K_{2,t}$. We have the following.

$(1)$  $|S\cap W|=3$. When $k=3$, there is a rainbow $S$-tree if and only if there exists $i\in \{1,2\}$ such that
$c(u_iv_j)$ are distinct ($j=1,2,3$); when $k\geq 4$, if there are  at least $4$ colors used by the color codes of three vertices, then there is a rainbow $S$-tree.

$(2)$  $|S\cap W|=2$. If both $i$-th ($i=1,2$) elements of two color codes are  distinct or at least three colors are used, then there is a rainbow $S$-tree.

$(3)$  $|S\cap W|=1$. If  $a_1\neq a_2$ for any color code $(a_1, a_2)$, then there is a rainbow $S$-tree.
\end{lem}
\begin{proof}
For $(1)$, firstly, when $k=3$, if there exists $i\in \{1,2\}$ such that
$c(u_iv_j)$ are distinct ($j=1,2,3$), then we find a rainbow $S$-tree
$T=\{u_iv_1,~u_iv_2,~u_iv_3\}$. And if there exists no $i\in \{1,2\}$ satisfying above condition, then we need to add at least two other vertices to obtain the rainbow $S$-tree, which implies there are at least four edges in rainbow $S$-tree. It contradicts that $k=3$.
Secondly, for $k\geq 4$, let code($v_1$)=$(c(v_1u_1),~c(v_1u_2))$,  code($v_2$)=$(c(v_2u_1),~c(v_2u_2))$, code($v_3$)=$(c(v_3u_1),~c(v_3u_2))$. If
there exists $i\in \{1,2\}$ such that
$c(u_iv_j)$ are distinct ($j=1,2,3$), the conclusion clearly holds.  And if not,
without loss of generality, assume $c(v_1u_1)= c(v_2u_1)\neq c(v_3u_1)$, then
we can find a rainbow $S$-tree $T=\{v_1u_2,v_2u_1,v_3u_1,v_3u_2\}$ or $T=\{v_1u_1,v_2u_2,v_3u_1,v_3u_2\}$.

For $(2)$, suppose that $v_1=u_1\in U,~v_2=w_1\in W,~v_3=w_2\in W$.
we can easily find a rainbow $S$-tree $T=\{u_1w_1,~u_1w_2\}$ with length $2$ or $T=\{u_1w_1,~w_1u_2,~u_2w_2\}$ with length $3$.

For $(3)$, suppose that $v_1=u_1\in U,~v_2=u_2\in U,~v_3=w_1\in W$. Then the tree $T=\{u_1w_1,~w_1u_2\}$ is a rainbow tree containing $S$.
\end{proof}

\begin{lem}\label{lem3}
For $t=3,4,$ $rx_3(K_{2,t})=3$, and $rx_3(K_{2,t})\geq 4$ for $t\geq 5$.
\end{lem}
\begin{proof}

First, we show the latter of conclusion that $rx_3(k_{2,t})\geq 4$ for $t\geq 5$.
By contradiction.
We assume there exists $t\geq 5$ such that $rx_3(k_{2,t})=3$ by Lemma \ref{lem1}.

From Lemma \ref{lem2} (1) and (2), if $rx_3(K_{2,t})=3$, then  for any three color codes: code($w_1$), code($w_2$), code($w_3$),
there exists $i\in\{1,2\}$ such that $c(w_1u_i),~c(w_2u_i),~c(w_3u_i)$ are different. Moreover, there is no  same color code in this case.

Now we try to connect the problem to the game of chess. The only fact needed about
the game is that rooks are \textit{isolate} if and only if any three of them lie in the different rows or the different columns of the chessboard. We give each square on the board a pair ($i,j$) of coordinates. The integer $i$ designates the row number of the square and the integer $j$ designates the column number of the square, where $i$ and $j$ are integers between $1$ and $3$. Our concern is the maximum number of rooks which are isolate on the chess   since it is the upper bound of $t$ with $rx_3(K_{2,t})=3$. We consider the condition from two factors:

 (a) if the rooks lie in different rows and columns.

 (b) if two of them lie in the same rows or columns.

It is easy to  verify  that at most $4$ rooks are isolate, such as $(1,2),~(2,1),~(1,3),~(3,1)$ shown in Figure 1, a contradiction. Thus, the conclusion holds.

Second, we give the vertices of $K_{2,3}$ any three color codes shown in Figure 1 (b)  and give the vertices of $K_{2,4}$ all color codes shown in Figure 1 (b). It is easy to check that  corresponding coloring is a $3$-rainbow coloring.
So $rx_3(K_{2,3})=rx_3(K_{2,4})=3$.
\begin{figure}[h,t,b,p]
\begin{center}
\includegraphics[width=6cm]{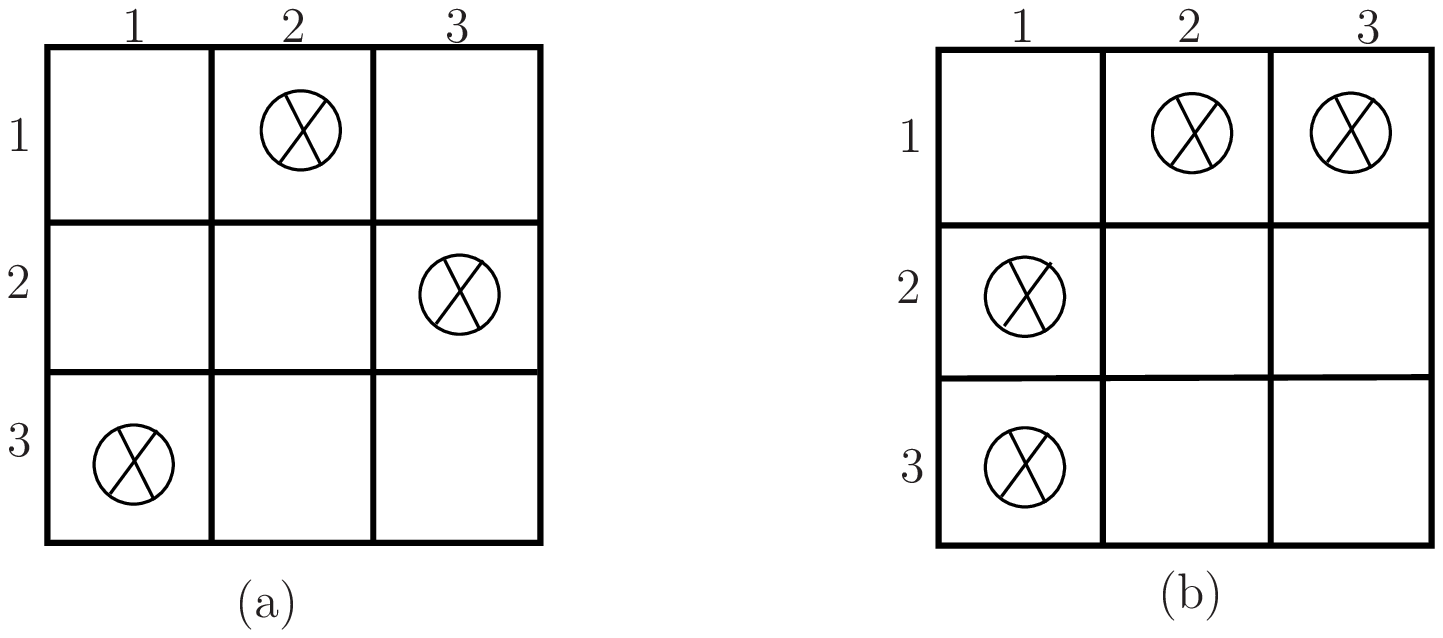}\\
Figure 1: An example of (a) and (b) used in Lemma \ref{lem3}.
\end{center}
\end{figure}
\end{proof}
From the proof of the Lemma \ref{lem3}, the following claim is easily obtained.
\begin{cla}\label{cla3}
If $|B|=3$, then $\beta_B=4$.
\end{cla}

\begin{lem}\label{lem4}
For $5\leq t\leq 8$, $rx_3(K_{2,t})=4$, and $rx_3(K_{2,t})\geq 5$ for $t\geq 9$.
\end{lem}
\begin{proof}
Similarly, we first prove the latter of the lemma.
By contradiction.
we may assume that there exists $t\geq 9$  such that $rx_3(k_{2,t})=4$. It follows that $\beta_B\geq 9$.
\begin{figure}[h,t,b,p]
\begin{center}
\includegraphics[width=3cm]{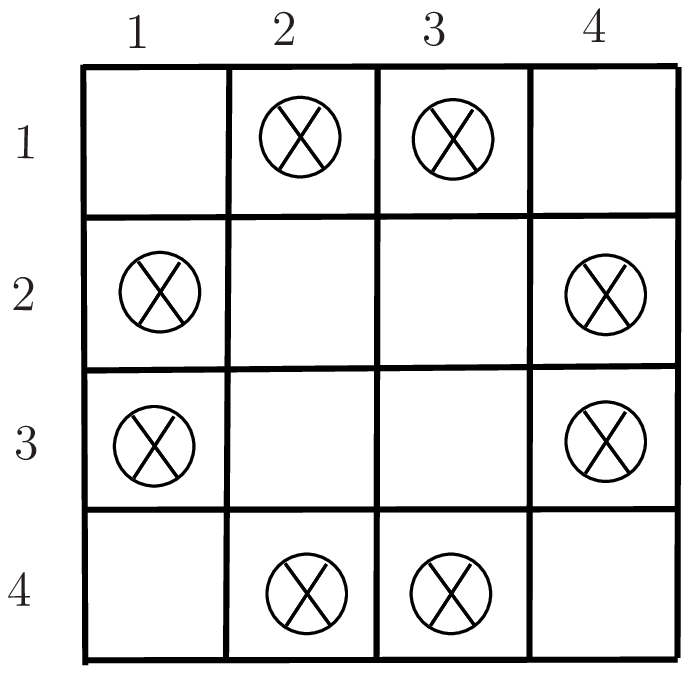}\\
Figure 2: The graph used in Lemma 2.4.
\end{center}
\end{figure}
Let $B=\{1,2,3,4\}$ be a set of 4 colors. Let $B_1=\{1,2,3\},~B_2=\{1,2,4\},~B_3=\{1,3,4\},~B_4=\{2,3,4\}$. Then $|B_i|=3$, so $\beta_{B_i}=4~(i=1,2,3,4$).  Since $B$ is the union of four $B_i~(i=1,2,3,4)$,   thus $\beta_B \leq 16$. And we find that a color code is limited in at least two $B_i~(i=1,2,3,4)$. So we get $\beta_B\leq 8$, a contradiction.

Then, we will get eight color codes such that the corresponding coloring is $3$-rainbow coloring.
We can seek eight rooks on the $4$-by-$4$ board, shown in Figure $2$.
By the Lemma \ref{lem2}, for any $t$ ($5\leq t \leq 8 $) rooks in Figure $2$,
we can find  a $3$-rainbow coloring of $K_{2,t}$.
Thus $rx_3(K_{2,t})=4$,~($5\leq t\leq 8$).
\end{proof}

\begin{lem}\label{lem5}
For $9\leq t\leq 20$, $rx_3(K_{2,t})=5$, and $rx_3(K_{2,t})\geq 6$ for $t\geq 21$.
\end{lem}

\begin{proof}

From the Claim $2$, we know $t\leq C_5^2\times 2=5\times4=20$, if $rx_3(k_{2,t})=5$. That is, $rx_3(k_{2,t})\geq 6$ for $t\geq 21$.

Next, we give $t$ vertices  $t$ color codes $(9\leq t\leq 20)$  and the corresponding coloring is $3$-rainbow coloring.
When $9\leq t\leq 10$, we just give $t$ vertices the first $t$ codes successively: (1,2),~(2,3),~(3,4),~(4,5),
~(3,1),
~(4,2),~(5,3),~(1,4),~(2,5),~(5,1) (see Figure 3). When $11\leq t\leq20$, we choose randomly  $t-10$ color codes from the remaining color codes in Figure 4(a) to give the $t-10$ vertices.

Then, it remains to show  the coloring is $3$-rainbow coloring. Let $S$ be a set of three vertices. By the Lemma \ref{lem2}, we can find a rainbow $S$-tree with the exception of the case:
$|S\cap W|=3$ and the only $3$ different colors used by the color codes of $S$.  Note that $3$ colors used by the color codes of $S$ may be allowed in this case. But there must exist a color code consisted of other two distinct colors. Thus we will find a rainbow $S$-tree with length $5$, for example see Figure 3.
 \begin{figure}[h,t,b,p]
\begin{center}
\includegraphics[width=3cm]{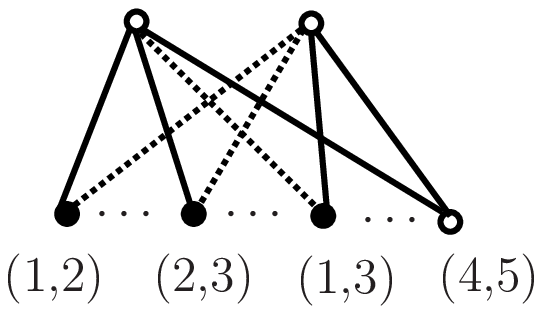}\\
Figure 3.
\end{center}
\end{figure}
When $t\geq 10$, if $3$ different colors used by the color codes of $S$, there must be a color code consisted of other two distinct colors appearing in the first $10$ color codes of $K_{2,t}$ by the strategy of coloring.  When $t=9$, we only to check the subcase that color codes of $S$ are limited in \{2,3,4\}. It is easy to verify the fact there is a rainbow tree connecting $S$, which correspond to the color codes $(2,3),~(3,4),~(4,2)$, respectively. So
the coloring is $3$-rainbow coloring. That is, for $9\leq t\leq 20$, $rx_3(k_{2,t})\leq 5$. With the aid of Lemma \ref{lem4}, we get $rx_3(k_{2,t})=5$, $9\leq t\leq 20$.
\begin{figure}[h,t,b,p]
\begin{center}
\includegraphics[width=10cm]{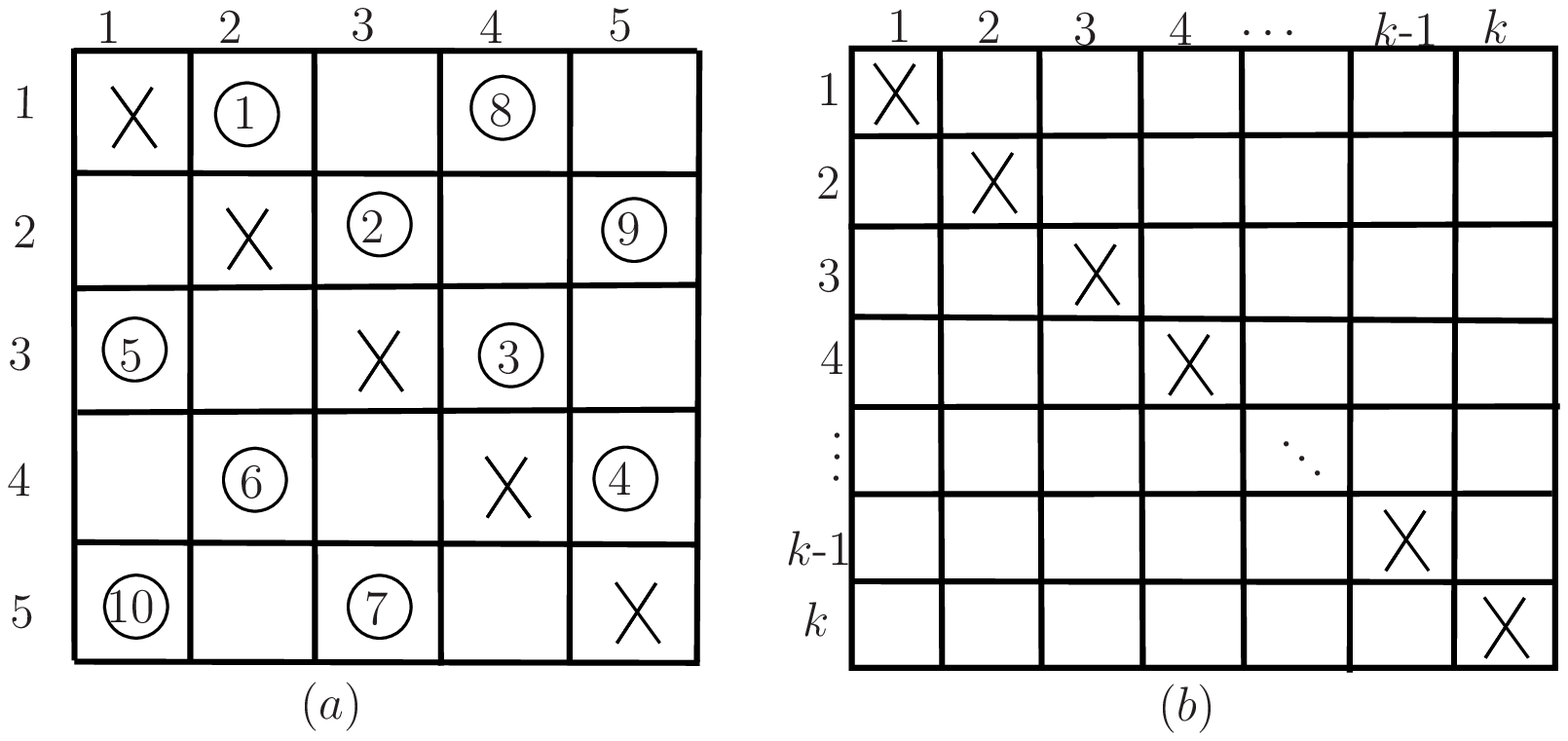}\\
Figure 4: The graph (a) used in Lemma 2.5 and the graph (b)used  in Lemma 2.6.
\end{center}
\end{figure}
\end{proof}
\begin{lem}\label{lem6}
For $(k-1)(k-2)+1\leq t\leq k(k-1)$, $rx_3(K_{2,t})=k$, $k\geq 6$.
\end{lem}
\begin{proof}

By the Claim $2$, if $rx_3(K_{2,t})=k$, then it has at most $C_k^2 \times 2=k(k-1)$  acceptable color codes. Thus when $t\geq k(k-1)+1$, $rx_3(K_{2,t})\geq k+1$. Similarly, when $t\geq (k-1)(k-2)+1$, $rx_3(K_{2,t})\geq k$.

Now we give a $3$-rainbow coloring of $K_{2,t}$ ($(k-1)(k-2)+1\leq t\leq k(k-1)$) with $k$ colors.  When $k\geq 6$, $(k-1)(k-2)+1 >\frac{1}{2} k(k-1)$, thus $t> \frac{1}{2}k(k-1) $.
We randomly give the first $\frac{1}{2}k(k-1)$ vertices $\frac{1}{2}k(k-1)$ color codes in
upper triangle of the chessboard, see Figure 4(b). Then the other $t-\frac{1}{2}k(k-1)$ vertices are received any $t-\frac{1}{2}k(k-1)$  remaining color codes in Figure 4(b).

Next we show this kind of coloring is $3$-rainbow coloring. Let $S$ be the set of three vertices. By Lemma \ref{lem2}, we only to check the case : $|S\cap W|=3$ and only $3$ different colors used by the color codes of three vertices. Similar to the proof of Lemma \ref{lem5}, we need to find  a color code consisted of other two distinct colors to construct a rainbow $S$-tree. Since the combinations of any two colors  have appeared  in first $k(k-1)/2$  color codes,  we can easily find such a color code. Hence, in any case, there is a rainbow $S$-tree with length at most $5$. That is, $(k-1)(k-2)+1\leq t\leq k(k-1)$, $rx_3(K_{2,t})\leq k$. So $rx_3(K_{2,t})= k$ for $(k-1)(k-2)+1\leq t\leq k(k-1)$.
\end{proof}
Now we complete the proof of Theorem 1.

\end{document}